\documentclass[12pt]{article}
\usepackage[a4paper, margin=2.3cm]{geometry}
\usepackage{amssymb,amsthm, amsmath, graphicx, yhmath}
\usepackage{hyperref}
\usepackage{color}
\usepackage{caption}
\usepackage{subcaption}

\usepackage{indentfirst}

\usepackage{pgf,tikz}
\usetikzlibrary{shapes,backgrounds,shapes,positioning,petri,topaths}
\usepackage{mathrsfs}
\usetikzlibrary{arrows}
\definecolor{ududff}{rgb}{0,0,0}
\definecolor{uuuuuu}{rgb}{0,0,0}
\definecolor{ffffff}{rgb}{256,256,256}
\definecolor{red}{rgb}{120,0,0}

\makeatletter
\DeclareFontFamily{U}{tipa}{}
\DeclareFontShape{U}{tipa}{m}{n}{<->tipa10}{}
\newcommand{\arc@char}{{\usefont{U}{tipa}{m}{n}\symbol{62}}}%

\newcommand{\arc}[1]{\mathpalette\arc@arc{#1}}

\newcommand{\arc@arc}[2]{%
  \sbox0{$\m@th#1#2$}%
  \vbox{
    \hbox{\resizebox{\wd0}{\height}{\arc@char}}
    \nointerlineskip
    \box0
  }%
}
\makeatother

\newtheorem{theorem}{Theorem}[section]

\newtheorem{lemma}[theorem]{Lemma}
\newtheorem{proposition}[theorem]{Proposition}

\newtheorem{question}[theorem]{Question}

\theoremstyle{definition}

\newtheorem{definition}[theorem]{Definition}

\newcommand{\R}{\mathbb R}

\author{Michael N. Manta}
\title{Triangle colorings require at least seven colors}

\begin{document}
\date{}
\maketitle
\begin{abstract}
We show that if a coloring of the plane has the properties that any two points at distance one are colored differently and the plane is partitioned into uniformly colored triangles under certain conditions, then it requires at least seven colors. This is also true for a coloring using uniformly colored polygons if it has a point bordering at least four polygons.
\end{abstract}

\section{Introduction}

The chromatic number of the plane is the minimum number of colors needed to color all points of $\R^2$ in such a way that any two points at distance one are colored differently. See Soifer \cite{Soi} for a detailed history of this problem. Although the problem was proposed by Nelson in the 1950s, the best lower bound has only recently been proven by De Grey \cite{Gre}, who showed that at least five colors are needed to color the plane. Exoo and Ismailescu \cite{Exo} provided an alternative proof of this result. On the other hand, the plane can be colored with seven colors by tiling the plane with hexagons or with squares (see Figure \ref{fig:tilings}), so the chromatic number of the plane is 5, 6, or 7.

It is natural to ask what the chromatic number is for colorings with certain restrictions. For instance, before De Grey's work, Falconer \cite{Fal} proved that five colors are needed for colorings that only use measurable sets. Woodall \cite{Woo} and Townsend \cite{Tow} proved that certain map-type colorings require at least six colors. Thomassen \cite{Tho} proved that seven colors are needed for certain restricted colorings. See Soifer \cite[Chapters 8, 9, and 24]{Soi} for more details about the chromatic number of restricted colorings.

Because the upper bound for the general problem follows from restricted colorings that only use polygons, it is natural to ask for the chromatic number of such colorings. Coulson \cite{Cou} proved that all polygon colorings require at least six colors, under the condition that all polygons are convex and have an area greater than some positive constant. Furthermore, Moustakis \cite{Mou} showed that at least seven colors are needed if the plane is tiled with congruent squares under certain conditions, which matches the square coloring in Figure \ref{fig:sqtile}.

\begin{figure}[ht]
        \centering
        \begin{subfigure}[b]{0.5\textwidth}
                \centering
                \begin{tikzpicture}[line cap=round,line join=round,>=triangle 45,x=1.0cm,y=1.0cm, scale=0.2]
\draw[line width=.9pt] (-1.25,2.165063509461097) -- (-2.5,0.) -- (-1.25,-2.1650635094610973) -- (1.25,-2.165063509461098) -- (2.5,0.) -- (1.25,2.1650635094610955) -- cycle;
\draw[line width=.9pt] (1.25,2.1650635094610955) -- (2.5,0.) -- (5.,0.) -- (6.25,2.1650635094610924) -- (5.,4.330127018922189) -- (2.5,4.330127018922193) -- cycle;
\draw[line width=.9pt] (-1.25,2.165063509461097) -- (1.25,2.1650635094610955) -- (2.5,4.330127018922195) -- (1.25,6.495190528383295) -- (-1.25,6.495190528383297) -- (-2.5,4.330127018922199) -- cycle;
\draw[line width=.9pt] (-1.25,2.165063509461097) -- (-2.5,4.330127018922199) -- (-5.,4.330127018922201) -- (-6.25,2.1650635094611) -- (-5.,0.) -- (-2.5,0.) -- cycle;
\draw[line width=.9pt] (-2.5,0.) -- (-5.,0.) -- (-6.25,-2.1650635094611097) -- (-5.,-4.330127018922215) -- (-2.5,-4.330127018922212) -- (-1.25,-2.1650635094611066) -- cycle;
\draw[line width=.9pt] (-1.25,-2.1650635094610973) -- (-2.5,-4.330127018922212) -- (-1.25,-6.495190528383315) -- (1.25,-6.4951905283833025) -- (2.5,-4.330127018922188) -- (1.25,-2.1650635094610853) -- cycle;
\draw[line width=.9pt] (1.25,-2.165063509461098) -- (2.5,-4.330127018922188) -- (5.,-4.330127018922157) -- (6.25,-2.1650635094610364) -- (5.,0.) -- (2.5,0.) -- cycle;

\draw[line width=.9pt] (5,4.330127018922212) -- (6.25,6.495190528383295);
\draw[line width=.9pt] (6.25,2.165063509461098) -- (7.5,2.165063509461098);
\draw[line width=.9pt] (6.25,-2.165063509461098) -- (7.5,-2.165063509461098);
\draw[line width=.9pt] (5,-4.330127018922212) -- (6.25,-6.495190528383295);

\draw[line width=.9pt] (-5,4.330127018922212) -- (-6.25,6.495190528383295);
\draw[line width=.9pt] (-6.25,2.165063509461098) -- (-7.5,2.165063509461098);
\draw[line width=.9pt] (-6.25,-2.165063509461098) -- (-7.5,-2.165063509461098);
\draw[line width=.9pt] (-5,-4.330127018922212) -- (-6.25,-6.495190528383295);

\draw (-1.15,1.4) node[anchor=north west] {1};
\draw (2.55,3.625) node[anchor=north west] {2};
\draw (2.55,-0.65) node[anchor=north west] {3};
\draw (-1.25,-2.8) node[anchor=north west] {4};
\draw (-4.95,-0.65) node[anchor=north west] {5};
\draw (-4.95,3.625) node[anchor=north west] {6};
\draw (-1.2,5.75) node[anchor=north west] {7};

\draw (6.25,5.75) node[anchor=north west] {6};
\draw (6.25,1.4) node[anchor=north west] {5};
\draw (6.25,-2.8) node[anchor=north west] {7};

\draw (-8.65,5.75) node[anchor=north west] {4};
\draw (-8.65,1.4) node[anchor=north west] {2};
\draw (-8.65,-2.8) node[anchor=north west] {3};

\draw (-4.95,-4.35) node[anchor=north west] {7};

\draw (2.45,-4.35) node[anchor=north west] {6};

\draw (-4.95,7.025) node[anchor=north west] {3};

\draw (2.45,7.025) node[anchor=north west] {4};

\end{tikzpicture}
                \caption{Hexagon coloring}
		\label{fig:hextile}
        \end{subfigure}%
                \begin{subfigure}[b]{0.5\textwidth}
                \centering
                
                \begin{tikzpicture}[line cap=round,line join=round,>=triangle 45,x=1.0cm,y=1.0cm, scale=0.17]
\draw[line width=.75pt] (-5.,0.) -- (0.,0.) -- (0.,5.) -- (-5.,5.) -- cycle;
\draw[line width=.75pt] (0.,0.) -- (5.,0.) -- (5.,5.) -- (0.,5.) -- cycle;
\draw[line width=.75pt] (-10.,0.) -- (-5.,0.) -- (-5.,5.) -- (-10.,5.) -- cycle;
\draw[line width=.75pt] (-15.,0.) -- (-10.,0.) -- (-10.,5.) -- (-15.,5.) -- cycle;
\draw[line width=.75pt] (5.,0.) -- (10.,0.) -- (10.,5.) -- (5.,5.) -- cycle;
\draw[line width=.75pt] (-20.,0.) -- (-15.,0.) -- (-15.,5.) -- (-20.,5.) -- cycle;
\draw[line width=.75pt] (-25.,0.) -- (-20.,0.) -- (-20.,5.) -- (-25.,5.) -- cycle;
\draw[line width=.75pt] (-22.5,0.) -- (-22.5,-5.) -- (-17.5,-5.) -- (-17.5,0.) -- cycle;
\draw[line width=.75pt] (-17.5,0.) -- (-17.5,-5.) -- (-12.5,-5.) -- (-12.5,0.) -- cycle;
\draw[line width=.75pt] (-12.5,0.) -- (-12.5,-5.) -- (-7.5,-5.) -- (-7.5,0.) -- cycle;
\draw[line width=.75pt] (-7.5,0.) -- (-7.5,-5.) -- (-2.5,-5.) -- (-2.5,0.) -- cycle;
\draw[line width=.75pt] (-2.5,0.) -- (-2.5,-5.) -- (2.5,-5.) -- (2.5,0.) -- cycle;
\draw[line width=.75pt] (2.5,0.) -- (2.5,-5.) -- (7.5,-5.) -- (7.5,0.) -- cycle;
\draw[line width=.75pt] (7.5,0.) -- (7.5,-5.) -- (12.5,-5.) -- (12.5,0.) -- cycle;
\draw[line width=.75pt] (-20.,-5.) -- (-25.,-5.) -- (-25.,-10.) -- (-20.,-10.) -- cycle;
\draw[line width=.75pt] (-15.,-5.) -- (-20.,-5.) -- (-20.,-10.) -- (-15.,-10.) -- cycle;
\draw[line width=.75pt] (-10.,-5.) -- (-15.,-5.) -- (-15.,-10.) -- (-10.,-10.) -- cycle;
\draw[line width=.75pt] (-5.,-5.) -- (-10.,-5.) -- (-10.,-10.) -- (-5.,-10.) -- cycle;
\draw[line width=.75pt] (0.,-5.) -- (-5.,-5.) -- (-5.,-10.) -- (0.,-10.) -- cycle;
\draw[line width=.75pt] (5.,-5.) -- (0.,-5.) -- (0.,-10.) -- (5.,-10.) -- cycle;
\draw[line width=.75pt] (10.,-5.) -- (5.,-5.) -- (5.,-10.) -- (10.,-10.) -- cycle;

\draw (-23.9,4.25) node[anchor=north west] {2};
\draw (-18.9,4.25) node[anchor=north west] {3};
\draw (-13.9,4.25) node[anchor=north west] {4};
\draw (-8.9,4.25) node[anchor=north west] {5};
\draw (-3.9,4.25) node[anchor=north west] {6};
\draw (1.1,4.25) node[anchor=north west] {7};
\draw (6.1,4.25) node[anchor=north west] {1};

\draw (-21.4,-0.75) node[anchor=north west] {5};
\draw (-16.4,-0.75) node[anchor=north west] {6};
\draw (-11.4,-0.75) node[anchor=north west] {7};
\draw (-6.4,-0.75) node[anchor=north west] {1};
\draw (-1.4,-0.75) node[anchor=north west] {2};
\draw (3.6,-0.75) node[anchor=north west] {3};
\draw (8.6,-0.75) node[anchor=north west] {4};

\draw (-23.9,-5.75) node[anchor=north west] {7};
\draw (-18.9,-5.75) node[anchor=north west] {1};
\draw (-13.9,-5.75) node[anchor=north west] {2};
\draw (-8.9,-5.75) node[anchor=north west] {3};
\draw (-3.9,-5.75) node[anchor=north west] {4};
\draw (1.1,-5.75) node[anchor=north west] {5};
\draw (6.1,-5.75) node[anchor=north west] {6};

\end{tikzpicture}
                \caption{Square coloring}
		\label{fig:sqtile}
        \end{subfigure}%
        \caption{In both colorings, polygons have unit length diagonals.}
        \label{fig:tilings}
\end{figure}
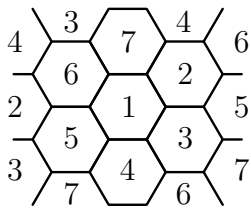
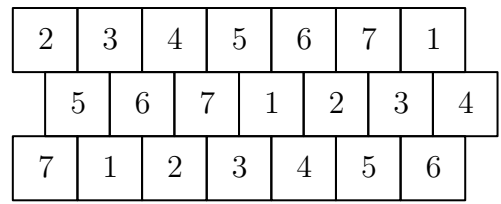

In this paper, we continue the study of restricted colorings that use polygons. In particular, we look at colorings that use triangles (see Figure \ref{fig:besttri} for an example), and we prove that at least seven colors are needed for any triangle coloring of the plane. We do this by showing that if any polygon coloring has a point that borders at least four polygons, then at least seven colors are needed. This, together with the observation that any triangle coloring has at least one point that borders at least four triangles, implies that triangle colorings require at least seven colors.

\section{Definitions and theorems}

We begin by defining the colorings using polygons that we study.

\begin{definition}
A \textit{polygon coloring} is a mapping from $\R^2$ to a finite set of colors, such that any two points a unit distance apart are colored differently, and the plane is partitioned into \textit{regions}, \textit{borders}, and \textit{vertices} with the following properties. Vertices are points, borders are open line segments whose endpoints are vertices, and regions are uniformly colored open polygons enclosed by borders and vertices. Any regions enclosed by the same vertex or border are colored differently. We only call a point a vertex if there exists a pair of borders it connects to that forms an angle other than $\pi$. A coloring is locally finite in the sense that any disk intersects finitely many regions, borders, and vertices.
\end{definition} 

Note that all borders and vertices may be colored arbitrarily, meaning that they may be colored independently of the regions that they border. The arguments in this paper are not affected by how the border points and vertices are colored. Furthermore, regions do not have to be strictly convex. 

Regions can \textit{share} borders or vertices, and borders can share vertices, even though as open sets they do not contain them. The \textit{degree} of a vertex is the number of borders that are connected to it. We say that a color \textit{occurs} at a border point or vertex $P$ if a region bordering $P$ contains that color. If $n$ colors occur at a vertex, then we call it an \textit{$n$-colored vertex}. Since in our definition regions that share a vertex must be colored differently, if a vertex has degree $n$ then it is $n$-colored. 

We now state the first of our two main theorems, which applies to all polygon colorings.

\begin{theorem}\label{thm:genthm}
If a polygon coloring of the plane contains a vertex of degree at least $4$, then at least $7$ colors are required.
\end{theorem} 

We define a \textit{triangle coloring} of the plane as a polygon coloring in which all the regions are triangles. We now state our second main theorem that applies to all triangle colorings, which is a consequence of Theorem \ref{thm:genthm}.

\begin{theorem}\label{thm:trithm}
Any triangle coloring of the plane requires at least $7$ colors.
\end{theorem} 

\section{Definitions and lemmas}

In this section, we introduce some definitions and lemmas that we use in our proof. These all refer to a polygon coloring with a fixed vertex $O$ and the circle $C_O$ of unit radius centered at $O$. 

\subsection{Crossings}

The proof relies on considering certain border points and vertices on $C_O$ which we call crossings.

\begin{definition}
A \textit{crossing} is either a point on $C_O$ that lies on a border not tangent to $C_O$ or a vertex on $C_O$ connected to at least one border that has points inside $C_O$ and at least one border that has points outside $C_O$.
\end{definition} 

\begin{figure}[ht]
        \centering
        \begin{subfigure}[b]{0.5\textwidth}
                \centering
                 \begin{tikzpicture}[line cap=round,line join=round,>=triangle 45,x=1.0cm,y=1.0cm, scale=0.5]
\draw [line width=1.pt] (0.,0.) circle (5.cm);
\draw [line width=1.pt] (-1.1255025201300173,0.8193951095066375)-- (0.,0.);
\draw [line width=1.pt] (0.,0.)-- (0.8530955439712657,0.9136140649400318);
\draw [line width=1.pt] (0.,0.)-- (-0.9056582907854304,-0.9707650437278539);
\draw [line width=1.pt] (0.,0.)-- (1.0729397733158528,-1.0963903176390464);

\draw [fill=black] (2.5,4.330127018922193) circle (3pt);
\draw [line width=1.pt] (1.6226830776696684,3.4397254141432962)-- (3.377316922330328,5.220528623701084);

\begin{scriptsize}
\draw [fill=uuuuuu] (0.,0.) circle (2.0pt);
\draw [fill=uuuuuu] (-1.1255025201300173,0.8193951095066375) circle (2.0pt);
\draw [fill=ududff] (0.8530955439712657,0.9136140649400318) circle (2pt);
\draw [fill=ududff] (-0.9056582907854304,-0.9707650437278539) circle (2pt);
\draw [fill=ududff] (1.0729397733158528,-1.0963903176390464) circle (2pt);

\draw [fill=black] (-2.5,4.330127018922193) circle (3pt);
\draw [line width=1.pt] (-1.3946013658158076,3.7465228212473294)-- (-2.5,4.330127018922193);
\draw [line width=1.pt] (-2.5,4.330127018922193)-- (-2.5333825456094092,3.0805728561710795);
\draw [line width=1.pt] (-2.5,4.330127018922193)-- (-3.0572006030941794,5.449067360455862);

\draw[color=uuuuuu] (-2.3,4.930127018922193) node {$P_1$};

\draw[color=uuuuuu] (2.4,4.930127018922193) node {$P_2$};

\draw[color=uuuuuu] (0, 0.5) node {$O$};
\end{scriptsize}
\end{tikzpicture}
                \caption{Points $P_1$ and $P_2$ are crossings}
		\label{fig:crossings}
        \end{subfigure}%
                \begin{subfigure}[b]{0.5\textwidth}
                \centering
                \begin{tikzpicture}[line cap=round,line join=round,>=triangle 45,x=1.0cm,y=1.0cm, scale=0.5]
\draw [line width=1.pt] (0.,0.) circle (5.cm);
\draw [line width=1.pt] (-1.1255025201300173,0.8193951095066375)-- (0.,0.);
\draw [line width=1.pt] (0.,0.)-- (0.8530955439712657,0.9136140649400318);
\draw [line width=1.pt] (0.,0.)-- (-0.9056582907854304,-0.9707650437278539);
\draw [line width=1.pt] (0.,0.)-- (1.0729397733158528,-1.0963903176390464);

\begin{scriptsize}
\draw [fill=black] (0.,0.) circle (2.0pt);
\draw [fill=uuuuuu] (-1.1255025201300173,0.8193951095066375) circle (2.0pt);
\draw [fill=ududff] (0.8530955439712657,0.9136140649400318) circle (2pt);
\draw [fill=ududff] (-0.9056582907854304,-0.9707650437278539) circle (2pt);
\draw [fill=ududff] (1.0729397733158528,-1.0963903176390464) circle (2pt);

\draw [fill=black] (-2.5,4.330127018922193) circle (3pt);
\draw [line width=1.pt] (-2.5,4.330127018922193)-- (-0.7679491924311237,5.330127018922193);
\draw [line width=1.pt] (-2.5,4.330127018922193)-- (-4.232050807568877,3.330127018922193);

\draw [fill=black] (2.5,4.330127018922193) circle (3pt);
\draw [line width=1.pt] (2.5,4.330127018922193)-- (0.7679491924311237,5.330127018922193);
\draw [line width=1.pt] (2.5,4.330127018922193)-- (4.232050807568877,3.330127018922193);
\draw [line width=1.pt] (2.5,4.330127018922193)-- (3.6963144919952007,4.532884529041009);

\draw[color=uuuuuu] (0, 0.5) node {$O$};

\draw[color=uuuuuu] (-2.6,4.830127018922193) node {$P_3$};

\draw[color=uuuuuu] (2.5,4.830127018922193) node {$P_4$};

\end{scriptsize}
\end{tikzpicture}
        \caption{Points $P_3$ and $P_4$ are pseudo-crossings}
		\label{fig:noncrossings}
        \end{subfigure}%
        \caption{Examples of crossings and pseudo-crossings on $C_O$}
        \label{fig:specialpoints}
\end{figure}
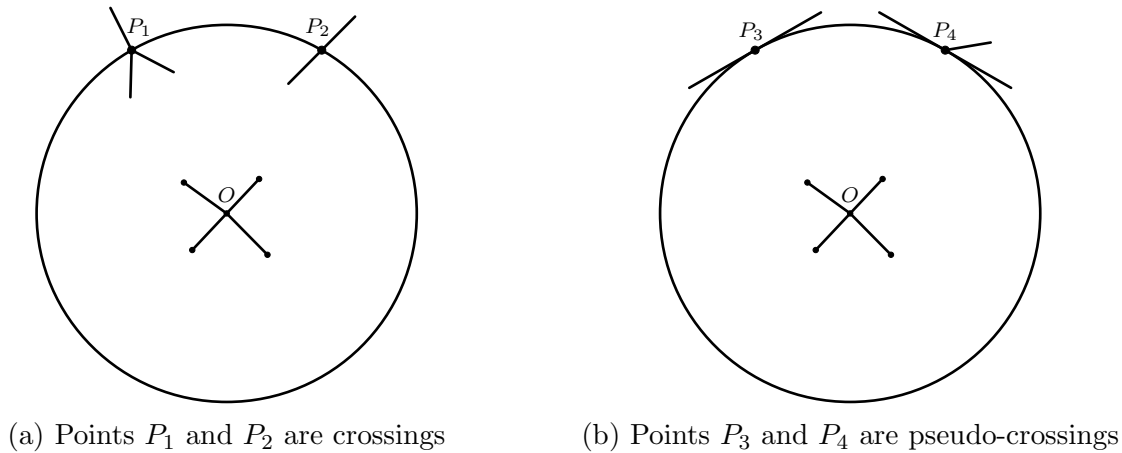

Note that a point on $C_O$ which lies on a border that is tangent to $C_O$ is not a crossing. Also, a vertex on $C_O$ connected to borders that do not have any points inside $C_O$ (see Figure \ref{fig:noncrossings}) or borders that do not have any points outside $C_O$ is not a crossing. We call these points \textit{pseudo-crossings}.

We now record the obvious fact that $C_O$ must contain at least one crossing (in fact, it is not hard to show that there must be at least six).

\begin{lemma}\label{lem:arclem}
The circle $C_O$ must contain at least one crossing.
\end{lemma} 
\begin{proof}
Assume that there are no crossings on $C_O$. This implies that all but finitely many points on $C_O$ lie within the same region, with pseudo-crossings being the only possible exceptions. Since the polygon coloring is locally finite, there must be two points on $C_O$ at distance one that are not exceptions. Thus, these two points are in the same region and must be colored the same. However, this is a contradiction since points at distance one from each other must be colored differently. 
\end{proof}

Now we state the following lemma about points that are not crossings on $C_O$.

\begin{lemma}\label{lem:neighlem}
If a point $P$ on $C_O$ is not a crossing, then there is a region that $P$ lies in or borders that cannot be colored any of the colors that occur at $O$.
\end{lemma} 
\begin{proof}
If $P$ lies within a region, consider a small circular neighborhood of $P$. Within each of the regions bordering $O$, there is a point a unit distance away from a point in the neighborhood of $P$. Each of these points must be colored differently, so the region $P$ lies in cannot be colored any of the colors that occur at $O$. On the other hand, if $P$ does not lie within a region and is a pseudo-crossing, then we can apply the previous argument to a point on $C_O$ in a region that $P$ borders.
\end{proof}

\subsection{Point types and special arcs}
The following definitions and lemmas all refer to a polygon coloring with a fixed 4-colored vertex $O$ and the circle $C_O$ of unit radius centered at $O$.

We define an \textit{inside neighborhood} of a point $P$ on $C_O$ as an open semi-circular neighborhood centered at $P$ that lies on the side closer to $O$ of the tangent line to $C_O$ at $P$. Similarly, we define an \textit{outside neighborhood} of a point $P$ on $C_O$ as an open semi-circular neighborhood centered at $P$ that lies on the side away from $O$ of the tangent line to $C_O$ at $P$. We say that if any points within a neighborhood of a point cannot be colored any of $n$ colors, then those $n$ colors are \textit{excluded} in that neighborhood.

\begin{definition}
An \textit{inward point} is a point $P$ on $C_O$ that has an inside neighborhood in which all $4$ colors occurring at $O$ are excluded (see Figure \ref{fig:inpoint}). 
\end{definition}

\begin{definition}
An \textit{outward point} is a point $P$ on $C_O$ that has an outside neighborhood in which all $4$ colors occurring at $O$ are excluded (see Figure \ref{fig:outpoint}). 
\end{definition} 

\begin{definition}
An \textit{alternative point} is a point $P$ on $C_O$ that has an outside neighborhood  in which only $3$ of the $4$ colors occurring at $O$ are excluded (see Figure \ref{fig:altpoint}). 
\end{definition} 

We can determine whether a point $P$ is inward, outward, or alternative as follows. Take the tangent line at $O$ to the unit circle $C_P$ centered at $P$ and count the borders on each side of the tangent. Adding one to this number usually gives the number of excluded colors on the corresponding side of $P$. The only exception to this heuristic is when any borders lie on the tangent line (see Figure \ref{fig:collinearpair} for an example) or all borders lie on one side of the tangent line.

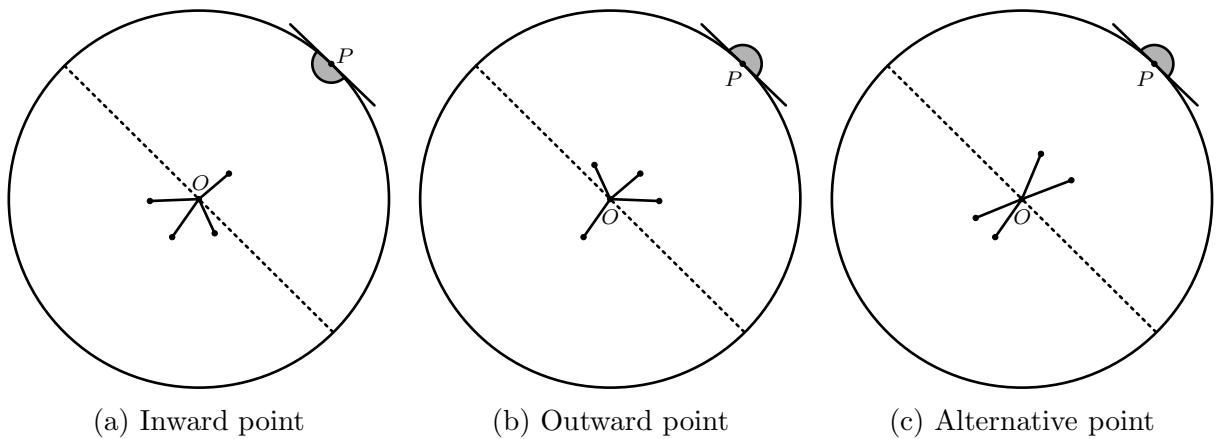
\begin{figure}[ht]
        \centering
        \begin{subfigure}[b]{0.33\textwidth}
                \centering

                \begin{tikzpicture}[line cap=round,line join=round,>=triangle 45,x=1.0cm,y=1.0cm,  scale=0.5]
\draw [line width=1pt] (0.,0.) circle (5.cm);
\draw [line width=1pt] (2.403815059142447,4.635067688461732)-- (4.618286161413787,2.4850613702588413);
\draw [shift={(3.482932991091283,3.5873636252222783)},line width=1pt,fill=black,fill opacity=0.3]  plot[domain=-0.7706289045528925:2.3709637490368998,variable=\t]({-0.5*cos(\t r)},{-0.5*sin(\t r)});

\draw [line width=1.pt, dotted] (-3.5355339059327373,3.5355339059327378)-- (3.5355339059327373,-3.5355339059327378);

\draw [line width=1.pt] (0.,0.)-- (0.7907474566524066,0.6769231762356426);
\draw [line width=1.pt] (0.,0.)-- (0.41556302464404626,-0.906696601244491);
\draw [line width=1.pt] (0.,0.)-- (-1.288589877504911,-0.05069266962571086);
\draw [line width=1.pt] (0.,0.)-- (-0.7027798059051065,-1.0080819404959127);

\begin{scriptsize}
\draw [fill=uuuuuu] (0.,0.) circle (2.0pt);
\draw [fill=uuuuuu] (3.482932991091283,3.5873636252222783) circle (2.0pt);
\draw [fill=ududff] (0.7907474566524066,0.6769231762356426) circle (2pt);
\draw [fill=ududff] (0.41556302464404626,-0.906696601244491) circle (2pt);
\draw [fill=ududff] (-1.288589877504911,-0.05069266962571086) circle (2pt);
\draw [fill=ududff] (-0.7027798059051065,-1.0080819404959127) circle (2pt);

\draw (3.4,4.2) node[anchor=north west] {$P$};
\draw (-0.4,0.85) node[anchor=north west] {$O$};
\end{scriptsize}
\end{tikzpicture}
                \caption{Inward point}
		\label{fig:inpoint}
        \end{subfigure}%
                \begin{subfigure}[b]{0.33\textwidth}
                \centering
               \begin{tikzpicture}[line cap=round,line join=round,>=triangle 45,x=1.0cm,y=1.0cm,  scale=0.5]
\draw [line width=1.pt] (0.,0.) circle (5.cm);
\draw [line width=1.pt] (2.403815059142447,4.635067688461732)-- (4.618286161413787,2.4850613702588413);
\draw [shift={(3.482932991091283,3.5873636252222783)},line width=1.pt,fill=black,fill opacity=0.3]  plot[domain=-0.7706289045528925:2.3709637490368998,variable=\t]({0.5*cos(\t r)},{0.5*sin(\t r)});

\draw [line width=1.pt, dotted] (-3.5355339059327373,3.5355339059327378)-- (3.5355339059327373,-3.5355339059327378);

\draw [line width=1.pt] (0.,0.)-- (0.7907474566524066,0.6769231762356426);
\draw [line width=1.pt] (0.,0.)-- (-0.41556302464404626,0.906696601244491);
\draw [line width=1.pt] (0.,0.)-- (1.288589877504911,-0.05069266962571086);
\draw [line width=1.pt] (0.,0.)-- (-0.7027798059051065,-1.0080819404959127);

\begin{scriptsize}
\draw [fill=uuuuuu] (0.,0.) circle (2.0pt);
\draw [fill=uuuuuu] (3.482932991091283,3.5873636252222783) circle (2.0pt);
\draw [fill=ududff] (0.7907474566524066,0.6769231762356426) circle (2pt);
\draw [fill=ududff] (-0.41556302464404626,0.906696601244491) circle (2pt);
\draw [fill=ududff] (1.288589877504911,-0.05069266962571086) circle (2pt);
\draw [fill=ududff] (-0.7027798059051065,-1.0080819404959127) circle (2pt);

\draw (2.8,3.6) node[anchor=north west] {$P$};
\draw (-0.44,-0.05) node[anchor=north west] {$O$};
\end{scriptsize}
\end{tikzpicture}
                \caption{Outward point}
		\label{fig:outpoint}
        \end{subfigure}%
        \begin{subfigure}[b]{0.33\textwidth}
                \centering
               \begin{tikzpicture}[line cap=round,line join=round,>=triangle 45,x=1.0cm,y=1.0cm,  scale=0.5]
\draw [line width=1.pt] (0.,0.) circle (5.cm);
\draw [line width=1.pt] (2.403815059142447,4.635067688461732)-- (4.618286161413787,2.4850613702588413);
\draw [shift={(3.482932991091283,3.5873636252222783)},line width=1.pt,fill=black,fill opacity=0.3]  plot[domain=-0.7706289045528925:2.3709637490368998,variable=\t]({0.5*cos(\t r)},{0.5*sin(\t r)});

\draw [line width=1.pt, dotted] (-3.5355339059327373,3.5355339059327378)-- (3.5355339059327373,-3.5355339059327378);

\draw [line width=1.pt] (0.,0.)-- (0.5,1.2);
\draw [line width=1.pt] (0.,0.)-- (-1.21556302464404626,-0.5);
\draw [line width=1.pt] (0.,0.)-- (1.3,0.5);
\draw [line width=1.pt] (0.,0.)-- (-0.7027798059051065,-1.0080819404959127);

\begin{scriptsize}
\draw [fill=uuuuuu] (0.,0.) circle (2.0pt);
\draw [fill=uuuuuu] (3.482932991091283,3.5873636252222783) circle (2.0pt);
\draw [fill=ududff] (-1.21556302464404626,-0.5) circle (2pt);
\draw [fill=ududff] (0.5,1.2) circle (2pt);
\draw [fill=ududff] (1.3,0.5) circle (2pt);
\draw [fill=ududff] (-0.7027798059051065,-1.0080819404959127) circle (2pt);

\draw (2.8,3.6) node[anchor=north west] {$P$};
\draw (-0.44,-0.05) node[anchor=north west] {$O$};
\end{scriptsize}
\end{tikzpicture}
                \caption{Alternative point}
		\label{fig:altpoint}
        \end{subfigure}%
         \caption{The shaded areas are the inside and outside neighborhoods of $P$.}
\end{figure}

With these definitions, we can show that these are the only kinds of points on $C_O$.


\begin{lemma}\label{lem:threetypes}
Any point on $C_O$ is an inward, outward, or alternative point.
\end{lemma}
\begin{proof}
Let $P$ be a point on $C_O$ and $C_P$ be the unit circle centered at $P$. The circle $C_P$ passes through $O$. We examine whether the borders connected to $O$ within a neighborhood of $O$ lie inside or outside of $C_P$. There are five possibilities for how many of these borders lie inside and outside of $C_P$: all four outside, three outside and one inside, two outside and two inside, one outside and three inside, and all four inside.

In the first case of four borders outside of $C_P$ and the second case of three borders outside of $C_P$ and one border inside of $C_P$, $4$ regions are intersected by any unit circle centered at a point within a sufficiently small inside neighborhood of $P$. Thus, the $4$ colors occurring at $O$ are excluded in an inside neighborhood of $P$, which makes $P$ an inward point.

In the third case of two borders outside of $C_P$ and two borders inside of $C_P$, $3$ regions are intersected by any unit circle centered at a point within a sufficiently small outside neighborhood of $P$. Thus, only $3$ of the $4$ colors occurring at $O$ are excluded in an outside neighborhood of $P$, which makes $P$ an alternative point.

In the fourth case of three borders inside of $C_P$ and one border outside of $C_P$, $4$ regions are intersected by any unit circle centered at a point within a sufficiently small outside neighborhood of $P$. Thus, the $4$ colors occurring at $O$ are excluded in an outside neighborhood of $P$, which makes $P$ an outward point.

Lastly, in case of four borders inside of $C_P$, $4$ regions are intersected by any unit circle centered at a point within a sufficiently small outside neighborhood of $P$. Thus, the $4$ colors occurring at $O$ are excluded in an outside neighborhood of $P$, which makes $P$ an outward point. 

Note that a point on $C_O$ can only be of one type. Since these are the only possible cases, all points on $C_O$ must be inward, outward, or alternative points.
\end{proof}

With the understanding that all points on $C_O$ come in three different types, we can define arcs that consist of points of the same type.

\begin{definition}
A maximal arc of $C_O$ consisting of only inward points is an \textit{inward arc}. Similarly, a maximal arc consisting of only outward points is an \textit{outward arc}. Lastly, a maximal arc consisting of only alternative points in which the same three colors are excluded in outside neighborhoods of each point is called an \textit{alternative arc}. 
\end{definition}

Note that an arc could be a single point of $C_O$ and endpoints may or may not be included in an arc. 

With this definition, we can introduce the following lemma, which determines the maximum number of alternative arcs on $C_O$.

\begin{lemma}\label{lem:fouralt}
The circle $C_O$ has at most four alternative arcs on it.
\end{lemma}
\begin{proof}
A point $P$ is an alternative point if and only if a unit circle $C_P$ centered at $P$ intersects two opposite regions bordering $O$. This is because within a neighborhood of $O$ two borders connected to $O$ must lie inside $C_P$ and two borders must lie outside $C_P$. There are three cases that are considered for the borders connected to $O$ (see Figure \ref{fig:altarcs}). The case where only two borders are collinear and all other borders lie on one side of the pair, the case where two borders form an angle greater than $\pi$, and the case where the borders are oriented in any other way.

We first consider the case of one pair of collinear borders. In this case, there are three arcs, one of which is a single point, such that the unit circle centered at any point on each arc intersects two opposite regions bordering $O$.

We next consider the case of two borders forming an angle greater than $\pi$. In this case, there are only two alternative arcs such that the unit circle centered at any point on each arc intersects two opposite regions bordering $O$.

Lastly, we consider the general case. In this case, there are four alternative arcs because each pair of opposite regions forces two alternative arcs.

Therefore, there are at most four alternative arcs on $C_O$.
\end{proof}

\begin{figure}[ht]
        \centering
        \begin{subfigure}[b]{0.33\textwidth}
                \centering
                
                \begin{tikzpicture}[line cap=round,line join=round,>=triangle 45,x=1.0cm,y=1.0cm, scale=0.5]
\draw [line width=1.pt, color=gray] (0.,0.) circle (5.cm);
\draw [line width=1.pt] (-2.,0.)-- (2.,0.);
\draw [line width=1.pt] (-1.4142135623730951,1.414213562373095)-- (0.,0.);
\draw [line width=1.pt] (0.,0.)-- (1.4142135623730951,1.414213562373095);
\draw [shift={(0.,5.)},line width=1.pt, dotted]  plot[domain=3.141592653589793:6.283185307179586,variable=\t]({1.*5.*cos(\t r)+0.*5.*sin(\t r)},{0.*5.*cos(\t r)+1.*5.*sin(\t r)});
\draw [shift={(0.,0.)},line width=2pt,color=uuuuuu]  plot[domain=-0.7853981633974483:0.7853981633974484,variable=\t]({1.*5.*cos(\t r)+0.*5.*sin(\t r)},{0.*5.*cos(\t r)+1.*5.*sin(\t r)});
\draw [shift={(0.,0.)},line width=2pt,color=uuuuuu]  plot[domain=2.356194490192345:3.9269908169872414,variable=\t]({1.*5.*cos(\t r)+0.*5.*sin(\t r)},{0.*5.*cos(\t r)+1.*5.*sin(\t r)});
\begin{scriptsize}
\draw [fill=uuuuuu] (0.,0.) circle (2.0pt);
\draw [fill=uuuuuu] (-2.,0.) circle (2pt);
\draw [fill=uuuuuu] (2.,0.) circle (2pt);
\draw [fill=uuuuuu] (1.4142135623730951,1.414213562373095) circle (2pt);
\draw [fill=uuuuuu] (-1.4142135623730951,1.414213562373095) circle (2pt);

\draw [fill=uuuuuu] (0.,5.) circle (3.0pt);
\draw [fill=uuuuuu] (3.5355339059327373,3.5355339059327378) circle (3pt);
\draw [fill=white] (3.5355339059327373,-3.5355339059327378) circle (3pt);

\draw [fill=white] (-3.5355339059327373,-3.5355339059327378) circle (3pt);
\draw [fill=uuuuuu] (-3.5355339059327373,3.5355339059327378) circle (3pt);

\draw (-.35,0) node[anchor=north west] {$O$};

\end{scriptsize}
\end{tikzpicture}

                \caption{One collinear pair}
		\label{fig:collinearpair}
        \end{subfigure}%
        \begin{subfigure}[b]{0.33\textwidth}
                \centering
                
                \begin{tikzpicture}[line cap=round,line join=round,>=triangle 45,x=1.0cm,y=1.0cm, scale=0.5]
\draw [line width=1.pt,color=gray] (0.,0.) circle (5.cm);
\draw [line width=1.pt] (0.,0.)-- (1.5850528510439255,1.2196751450273648);
\draw [line width=1.pt] (0.,0.)-- (-1.521425946097175,1.2981768333098214);
\draw [line width=1.pt] (0.,0.)-- (-0.6400020690936958,1.8948343863134287);
\draw [line width=1.pt] (0.,0.)-- (0.7561177966425573,1.8515630903645721);
\draw [shift={(0.,5.)},line width=1.pt, dotted]  plot[domain=3.141592653589793:6.283185307179586,variable=\t]({1.*5.*cos(\t r)+0.*5.*sin(\t r)},{0.*5.*cos(\t r)+1.*5.*sin(\t r)});
\draw [shift={(0.,0.)},line width=2.pt]  plot[domain=-0.38769867977257366:0.3257305792598854,variable=\t]({1.*5.*cos(\t r)+0.*5.*sin(\t r)},{0.*5.*cos(\t r)+1.*5.*sin(\t r)});
\draw [shift={(0.,0.)},line width=2.pt]  plot[domain=2.7538939738172195:3.4673232328496786,variable=\t]({1.*5.*cos(\t r)+0.*5.*sin(\t r)},{0.*5.*cos(\t r)+1.*5.*sin(\t r)});

\begin{scriptsize}
\draw [fill=black] (0.,0.) circle (2pt);
\draw [fill=black] (1.5850528510439255,1.2196751450273648) circle (2pt);
\draw [fill=black] (-1.521425946097175,1.2981768333098214) circle (2pt);
\draw [fill=black] (-0.6400020690936958,1.8948343863134287) circle (2pt);
\draw [fill=black] (0.7561177966425573,1.8515630903645721) circle (2pt);
\draw [fill=ffffff] (4.62890772591143,-1.8902944916063933) circle (3.0pt);
\draw [fill=uuuuuu] (4.737085965783573,1.60000517273424) circle (3.0pt);
\draw [fill=uuuuuu] (-4.62890772591143,1.8902944916063933) circle (3.0pt);
\draw [fill=ffffff] (-4.737085965783572,-1.6000051727342397) circle (3.0pt);
\draw (-.35,0) node[anchor=north west] {$O$};
\end{scriptsize}
\end{tikzpicture}

                \caption{Concave example}
		\label{fig:concave}
        \end{subfigure}%
        \begin{subfigure}[b]{0.33\textwidth}
                \centering
               \begin{tikzpicture}[line cap=round,line join=round,>=triangle 45,x=1.0cm,y=1.0cm, scale=0.5]
\draw [line width=1.pt, color=gray] (0.,0.) circle (5.cm);
\draw [line width=1.pt] (-1.0970414663416408,1.672273907327081)-- (0.,0.);
\draw [line width=1.pt] (0.,0.)-- (1.4643382564060556,1.362245745387253);
\draw [line width=1.pt] (0.,0.)-- (1.7953230740095407,-0.8813711249688936);
\draw [line width=1.pt] (0.,0.)-- (-1.1070833348303741,-1.665642965866592);
\draw [shift={(0.,5.)},line width=1.pt, dotted]  plot[domain=3.141592653589793:6.283185307179586,variable=\t]({1.*5.*cos(\t r)+0.*5.*sin(\t r)},{0.*5.*cos(\t r)+1.*5.*sin(\t r)});
\draw [shift={(0.,0.)},line width=2pt]  plot[domain=2.5549817831941235:3.722186694582359,variable=\t]({1.*5.*cos(\t r)+0.*5.*sin(\t r)},{0.*5.*cos(\t r)+1.*5.*sin(\t r)});
\draw [shift={(0.,0.)},line width=2pt]  plot[domain=1.1144340759794773:2.3200914943461477,variable=\t]({1.*5.*cos(\t r)+0.*5.*sin(\t r)},{0.*5.*cos(\t r)+1.*5.*sin(\t r)});
\draw [shift={(0.,0.)},line width=2pt]  plot[domain=-0.5866108703956696:0.5805940409925657,variable=\t]({1.*5.*cos(\t r)+0.*5.*sin(\t r)},{0.*5.*cos(\t r)+1.*5.*sin(\t r)});
\draw [shift={(0.,0.)},line width=2pt]  plot[domain=4.25602672956927:5.461684147935941,variable=\t]({1.*5.*cos(\t r)+0.*5.*sin(\t r)},{0.*5.*cos(\t r)+1.*5.*sin(\t r)});
\begin{scriptsize}
\draw [fill=uuuuuu] (0.,0.) circle (2.0pt);
\draw [fill=uuuuuu] (-1.0970414663416408,1.672273907327081) circle (2pt);
\draw [fill=uuuuuu] (1.4643382564060556,1.362245745387253) circle (2pt);
\draw [fill=uuuuuu] (1.7953230740095407,-0.8813711249688936) circle (2pt);
\draw [fill=uuuuuu] (-1.1070833348303741,-1.665642965866592) circle (2pt);

\draw [fill=white] (-4.164107414666481,2.767708337075936) circle (3.0pt);
\draw [fill=white] (-4.180684768317702,-2.742603665854102) circle (3.0pt);
\draw [fill=white] (-3.4056143634681324,3.660845641015139) circle (3.0pt);
\draw [fill=uuuuuu] (2.203427812422234,4.488307685023852) circle (3pt);
\draw [fill=uuuuuu] (4.16410741466648,-2.7677083370759354) circle (3pt);
\draw [fill=uuuuuu] (4.180684768317702,2.742603665854102) circle (3pt);
\draw [fill=white] (-2.203427812422234,-4.488307685023852) circle (3pt);
\draw [fill=uuuuuu] (3.405614363468133,-3.6608456410151398) circle (3pt);

\draw (-.35,0) node[anchor=north west] {$O$};

\end{scriptsize}
\end{tikzpicture}                
		\caption{Other orientation}
		\label{fig:otherorientation}
        \end{subfigure}%
         \caption{The dark arcs are alternative arcs and the dotted circles are examples of $C_P$.}
         \label{fig:altarcs}
\end{figure}
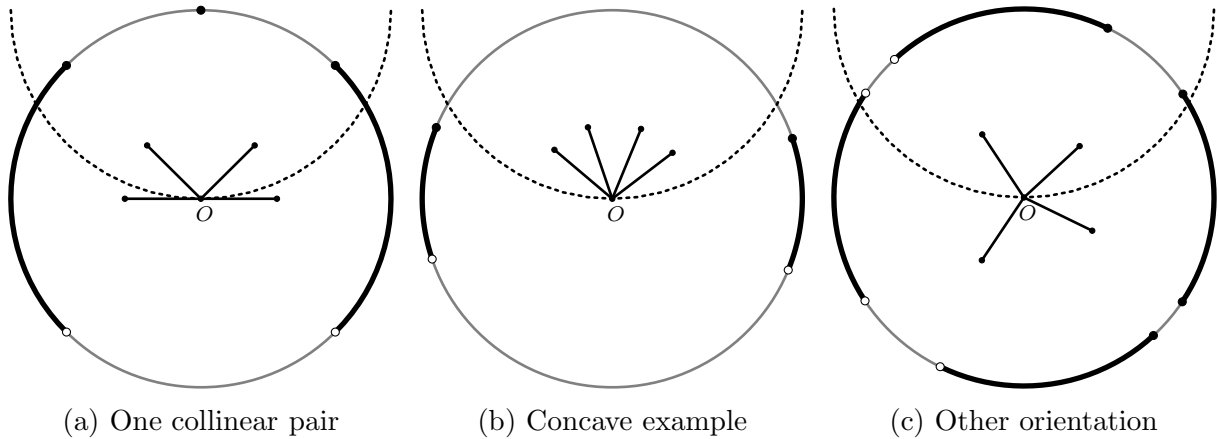

Note that the only two cases in which the number of alternative arcs on $C_O$ is fewer than four occurs when only two borders are collinear with all other borders on one side of the pair as shown in Figure \ref{fig:collinearpair} or when two borders form an angle greater than $\pi$ as shown in Figure \ref{fig:concave}.

\subsection{Lemmas about crossings}

The following lemmas all refer to a polygon coloring with a fixed 4-colored vertex $O$ and the circle $C_O$ of unit radius centered at $O$.

\begin{lemma}\label{lem:crossinglem}
The two regions that border a crossing $P$ on $C_O$ and contain points on $C_O$ must be colored a fifth and sixth color that do not occur at $O$.
\end{lemma}
\begin{proof}
Because the coloring is locally finite, there must be two regions that border $P$. Applying Lemma \ref{lem:neighlem} to a point in each region, we find that these regions cannot be colored any of the colors which occur at $O$. Since both regions cannot be colored the same fifth color by our definition, one region must be colored a fifth color and the other must be colored a sixth color.
\end{proof}

Using Lemma \ref{lem:crossinglem}, we can now introduce lemmas that prove that certain crossings necessitate seven colors for the coloring.

\begin{lemma}\label{lem:inoutlem}
If a crossing $P$ on $C_O$ is an inward or outward point, then at least seven colors are needed for the coloring.
\end{lemma}
\begin{proof}
Let $Q$ be a point at distance one from $P$ on $C_O$. Since a fifth and sixth color occur at $P$ by Lemma \ref{lem:crossinglem}, if $Q$ is not a crossing, then the region $Q$ lies in or borders must be colored a seventh color by Lemma \ref{lem:neighlem}. 

Hence, we consider the case where $Q$ is a crossing (see Figure \ref{fig:inout}). Since a fifth and sixth color occur at $Q$ by Lemma \ref{lem:crossinglem}, we consider the points on $C_Q$ within a neighborhood of $P$. Regardless of whether $P$ is an inward or outward point, there are points on $C_Q$ that lie inside a region and also in an inside or outside neighborhood of $P$. Since these points cannot be colored any of the six colors that occur at $O$ and $Q$ by Lemma \ref{lem:neighlem}, at least seven colors are needed.
\end{proof}

In addition to considering the case where crossings are inward or outward points, we also introduce a lemma that considers the case of crossings that are alternative points.

\begin{lemma}\label{lem:altlem}
If two crossings lie at distance one on the same alternative arc, then at least seven colors are needed for the coloring.
\end{lemma}
\begin{proof}
Let $P$ and $Q$ be the crossings and let $C_P$ and $C_Q$ be the unit circles centered at $P$ and $Q$ respectively. Since a fifth and sixth color occur at $Q$ by Lemma \ref{lem:crossinglem}, if the points on $C_Q$ within an outside neighborhood of $P$ lie within a region that $C_O$ passes through, then at least seven colors are needed by Lemma \ref{lem:neighlem} . 

Therefore, we consider the case in which points on $C_P$ and $C_Q$ within outside neighborhoods at $P$ and $Q$ lie within regions that $C_O$ does not pass through. Without loss of generality, consider a region bordering $Q$ that only $C_P$ passes through. If this region is colored a seventh color, then we are done. 

Suppose that this region is not colored a seventh color. Since $Q$ is an alternative point, this region can only be colored the fourth color not excluded from the colors occurring at $O$ (see Figure \ref{fig:altalt}). By Lemma \ref{lem:neighlem}, a region which borders $P$ and contains points on $C_Q$ cannot be colored the fourth, fifth, and six colors that occur at $Q$. Since $P$ lies on the same alternative arc as $Q$, the other three colors occurring at $O$ are also excluded in the region that borders $P$. Thus, at least seven colors are needed. \end{proof}

\begin{figure}[ht]
        \centering
        \begin{subfigure}[b]{0.5\textwidth}
                \centering
               \begin{tikzpicture}[line cap=round,line join=round,>=triangle 45,x=1.0cm,y=1.0cm, scale=0.5]
\draw [line width=1.pt] (0.,0.) circle (5.cm);
\draw [line width=1.pt] (-0.6816229033466099,1.0478025661513422)-- (0.,0.);
\draw [line width=1.pt] (0.,0.)-- (-1.2387888894394545,-0.16703917924056844);
\draw [line width=1.pt] (0.,0.)-- (0.4500586699083075,1.1661677382093731);
\draw [line width=1.pt] (0.,0.)-- (0.5506276500054399,-1.122189463080761);
\draw [line width=1.pt] (-3.254745078836226,5.3265505563662305)-- (-1.7452549211637716,3.3337034814781568);
\draw [line width=1.pt] (1.6504184119090541,3.4132240114143837)-- (2.5,4.330127018922193);
\draw [line width=1.pt] (2.5,4.330127018922193)-- (2.890644357935366,3.1427361878925497);
\draw [line width=1.pt] (2.5,4.330127018922193)-- (3.0412302605588284,5.456878901572239);
\draw [shift={(2.5,4.330127018922193)},line width=1.pt, dotted]  plot[domain=2.7862605813657595:5.4534430523534745,variable=\t]({1.*5.*cos(\t r)+0.*5.*sin(\t r)},{0.*5.*cos(\t r)+1.*5.*sin(\t r)});
\begin{scriptsize}
\draw [fill=uuuuuu] (0.,0.) circle (2.0pt);

\draw (-1.0699336262195465,0.5821198678064264) node[anchor=north west] {1};
\draw (-0.42442144355598228,1.1097017386609105) node[anchor=north west] {2};
\draw (0.4276430052507715,0.4934979134723176) node[anchor=north west] {3};
\draw (-0.60476464512641364,-0.1144271589525982) node[anchor=north west] {4};
\draw (1.9503886178142224,5.324844522833219) node[anchor=north west] {5};
\draw (2.6248659563823267,4.982088431501437) node[anchor=north west] {6};
\draw (-2.5633708813080254,5.175879366928679) node[anchor=north west] {7};

\draw[color=uuuuuu] (0.3752301154894207,0.18953537725985969) node {$O$};
\draw[color=uuuuuu] (-2.9675102576898646,4.34950394095487) node {$P$};
\draw[color=black] (1.9455628382915209,4.222951674811681) node {$Q$};

\draw [fill=black] (2.5,4.330127018922193) circle (3pt);
\draw [fill=uuuuuu] (-2.5,4.330127018922193) circle (3.0pt);
\draw [fill=black] (-0.6816229033466099,1.0478025661513422) circle (2pt);
\draw [fill=black] (-1.2387888894394545,-0.16703917924056844) circle (2pt);
\draw [fill=black] (0.4500586699083075,1.1661677382093731) circle (2pt);
\draw [fill=black] (0.5506276500054399,-1.122189463080761) circle (2pt);
\end{scriptsize}
\end{tikzpicture}
                \caption{Crossing $P$ is an outward point}
		\label{fig:inout}
        \end{subfigure}%
                \begin{subfigure}[b]{0.5\textwidth}
                \centering
                \begin{tikzpicture}[line cap=round,line join=round,>=triangle 45,x=1.0cm,y=1.0cm, scale=0.5]
\draw [line width=1.pt] (0.,0.) circle (5.cm);
\draw [line width=1.pt] (-0.6816229033466099,1.0478025661513422)-- (0.,0.);
\draw [line width=1.pt] (0.,0.)-- (-0.8590635758968608,-0.908025204808379);
\draw [line width=1.pt] (0.,0.)-- (0.4500586699083075,1.1661677382093731);
\draw [line width=1.pt] (0.,0.)-- (0.5506276500054399,-1.122189463080761);

\draw [line width=1.pt] (-3.254745078836226,5.3265505563662305)-- (-1.7452549211637716,3.3337034814781568);
\draw [line width=1.pt] (1.6504184119090541,3.4132240114143837)-- (2.5,4.330127018922193);
\draw [line width=1.pt] (2.5,4.330127018922193)-- (2.890644357935366,3.1427361878925497);
\draw [line width=1.pt] (2.5,4.330127018922193)-- (3.421857703399558,5.174329823159951);

\draw [line width=1.pt] (-2.5,4.330127018922193)-- (-2.124468540622823,5.522383754285029);
\draw [line width=1.pt] (2.5,4.330127018922193)-- (2.5081458592324717,5.580100476631341);

\draw [shift={(2.5,4.330127018922193)},line width=1.pt, dotted]  plot[domain=2.7862605813657595:5.4534430523534745,variable=\t]({1.*5.*cos(\t r)+0.*5.*sin(\t r)},{0.*5.*cos(\t r)+1.*5.*sin(\t r)});
\begin{scriptsize}
\draw [fill=uuuuuu] (0.,0.) circle (2.0pt);

\draw (-0.9699336262195465,0.4221198678064264) node[anchor=north west] {1};
\draw (-0.42442144355598228,1.1097017386609105) node[anchor=north west] {2};
\draw (0.4276430052507715,0.4934979134723176) node[anchor=north west] {3};
\draw (-0.45476464512641364,-0.1144271589525982) node[anchor=north west] {4};

\draw (1.8003886178142224,5.324844522833219) node[anchor=north west] {5};
\draw (2.4003886178142224,5.424844522833219) node[anchor=north west] {4};
\draw (2.7248659563823267,4.882088431501437) node[anchor=north west] {6};

\draw (-3.3633708813080254,4.875879366928679) node[anchor=north west] {6};
\draw (-3.0633708813080254,5.475879366928679) node[anchor=north west] {7};
\draw (-2.4033708813080254,5.275879366928679) node[anchor=north west] {5};

\draw[color=uuuuuu] (0.3752301154894207,0.18953537725985969) node {$O$};
\draw[color=uuuuuu] (-1.9875102576898646,4.23950394095487) node {$P$};
\draw[color=black] (1.9455628382915209,4.222951674811681) node {$Q$};

\draw [fill=black] (2.5,4.330127018922193) circle (3pt);
\draw [fill=uuuuuu] (-2.5,4.330127018922193) circle (3.0pt);
\draw [fill=black] (-0.6816229033466099,1.0478025661513422) circle (2pt);
\draw [fill=black] (-0.8590635758968608,-0.908025204808379) circle (2pt);
\draw [fill=black] (0.4500586699083075,1.1661677382093731) circle (2pt);
\draw [fill=black] (0.5506276500054399,-1.122189463080761) circle (2pt);
\end{scriptsize}
\end{tikzpicture}
                               \caption{Crossings $P$ and $Q$ are alternative points}
		\label{fig:altalt}
        \end{subfigure}%
        \caption{Examples of $P$ and $Q$ on $C_O$}
        \label{fig:sevencolors}
\end{figure}
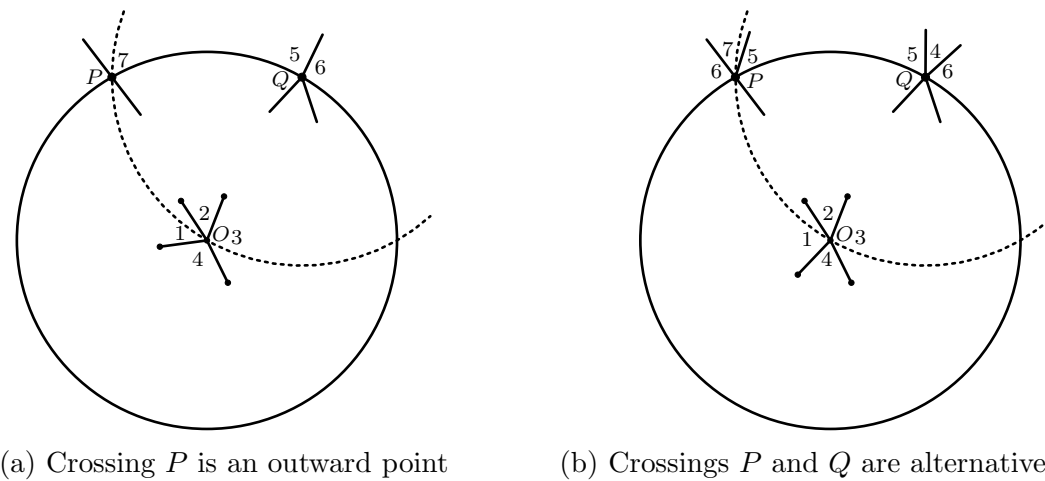

\section{Proof of Theorem \ref{thm:genthm}}

Let $\alpha$ be a polygon coloring of the plane and suppose it contains a vertex $O$ of degree 5 or more. By Lemma \ref{lem:arclem} there is at least one crossing $P$ on the unit circle $C_O$ centered at $O$. The regions that $C_O$ passes through which border $P$ cannot be colored any five colors that occur at $O$ by Lemma \ref{lem:neighlem}. Therefore, by our definition, the bordering regions must be colored using a sixth and seventh color. 

Hence, we consider the case where all vertices in $\alpha$ have degree at most $4$, and $\alpha$ contains a vertex $O$ of degree 4. Let $P$ be a crossing on $C_O$ and $Q$ be a point at distance one clockwise from $P$ on $C_O$. If $Q$ is not a crossing, then by Lemma \ref{lem:neighlem} the region $Q$ lies in or borders must be colored a seventh color. Therefore, we can assume that $Q$ is a crossing. Similarly, if we repeat the same reasoning clockwise starting from point $Q$, we can assume that all points on the inscribed hexagon on $C_O$ that contains $P$ and $Q$ are crossings.

If any of the vertices on the hexagon are inward or outward points, then by Lemma \ref{lem:inoutlem} at least seven colors are needed for $\alpha$. Hence, we can assume all vertices on the hexagon are alternative points. By Lemma \ref{lem:fouralt}, there are at most four alternative arcs on $C_O$. Therefore, by the pigeonhole principle, there must be two adjacent points on the inscribed hexagon that lie on the same alternative arc, so by Lemma \ref{lem:altlem} more than six colors are needed for $\alpha$.

Hence, if a polygon coloring has at least one vertex of degree at least 4, then at least seven colors are needed. 

\section{Triangle colorings}

In this section, we introduce an observation about triangle colorings that proves that Theorem \ref{thm:trithm} is a consequence of Theorem \ref{thm:genthm}. First, we define a type of line segment in triangle colorings.

\begin{definition}
A \textit{borderline} is a closed line segment consisting of borders and vertices. We call a borderline with vertex endpoints $A$ and $B$ borderline $AB$. 
\end{definition}

We now state a lemma regarding triangle colorings.

\begin{lemma}\label{lem:obtuselem}
For three borderlines $AB$, $BC$, and $CD$ such that rays $BA$ and $CD$ do not intersect and angles $\angle ABC$ and $\angle BCD$ are each less than $\pi$, there must be a vertex of degree at least 4 on borderline $BC$.
\end{lemma} 
\begin{proof}
We know that angle $\angle ABC$ lies within a triangle because otherwise vertex $B$ would have degree at least 4 or there would be an angle greater than $\pi$, which is not possible in a triangle coloring. Because $\angle ABC$ lies within a triangle with a vertex $E$ on the borderline $BC$, vertex $E$ can either be at $C$ or in between $B$ and $C$. If vertex $E$ lies at $C$, then $C$ has degree at least 4 because otherwise there would be an angle greater than $\pi$, which is not possible in a triangle coloring, or the ray $CD$ would intersect ray $BA$. Thus, we assume vertex $E$ lies in between $B$ and $C$. By a similar argument it follows that $\angle BCD$ lies within a triangle with a vertex $F$ on borderline $EC$. If this vertex lies at $E$, then $E$ has degree at least 4, so we assume $F$ lies in between $E$ and $C$. We can repeat the same reasoning for angles within triangles at $E$ and $F$ and so on, until two vertices $P$ and $Q$ are found on borderline $BC$ that do not have any vertices between them. Such a pair exists because the coloring is locally finite. Either $Q$ has degree at least 4, or there is a triangle containing the angle at $Q$, and then $P$ has degree at least 4.\end{proof}

\begin{figure}[ht]
        \centering
\begin{tikzpicture}[line cap=round,line join=round,>=triangle 45,x=1.0cm,y=1.0cm, scale=1.2]
\draw [line width=1.pt] (-4.,2.)-- (-2.,0.);
\draw [line width=1.pt] (-2.,0.)-- (2.,0.);
\draw [line width=1.pt] (2.,0.)-- (4.,3.);
\draw [line width=1.pt] (-4.,2.)-- (-1.,0.);
\draw [line width=1.pt] (4.,3.)-- (1.,0.);
\draw [line width=1.pt] (-2.5046153846153847,1.003076923076923)-- (-0.5,0.);
\draw [line width=1.pt] (0.5,0.)-- (3.,2.);
\draw [line width=1.pt] (-0.5,0.)-- (1.8853658536585365,1.1082926829268291);
\begin{scriptsize}
\draw [fill=uuuuuu] (-4.,2.) circle (1pt);
\draw[color=uuuuuu] (-3.86,2.2) node {$A$};
\draw [fill=uuuuuu] (-2.,0.) circle (1pt);
\draw[color=uuuuuu] (-1.86,0.17) node {$B$};
\draw [fill=uuuuuu] (2.,0.) circle (1pt);
\draw[color=uuuuuu] (1.9,0.25) node {$C$};
\draw [fill=uuuuuu] (4.,3.) circle (1pt);
\draw[color=uuuuuu] (4.14,3.2) node {$D$};
\draw [fill=uuuuuu] (-1.,0.) circle (1pt);
\draw[color=uuuuuu] (-1.10,-0.21) node {$E$};
\draw [fill=uuuuuu] (1.,0.) circle (1pt);
\draw[color=uuuuuu] (1.12,-0.21) node {$F$};
\draw [fill=uuuuuu] (-0.5,0.) circle (1pt);
\draw[color=uuuuuu] (-0.54,0.27) node {$P$};
\draw [fill=uuuuuu] (0.5,0.) circle (1pt);
\draw[color=uuuuuu] (0.36,-0.21) node {$Q$};
\end{scriptsize}
\end{tikzpicture}
\caption{This is an example in which vertex $P$ has degree 4.}\label{fig:obtusefig}
\end{figure}
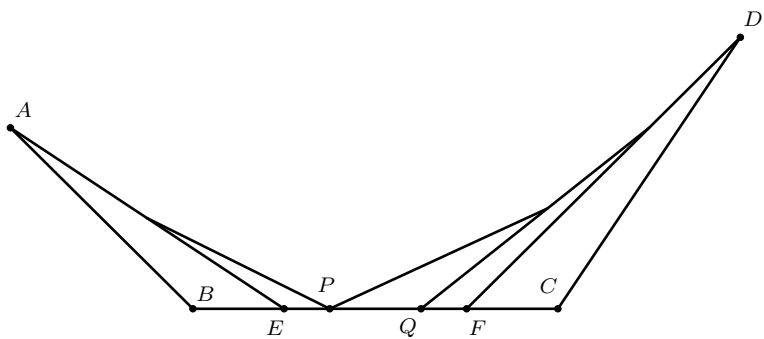

\begin{proposition}
Any triangle coloring has a vertex of degree 4 or more.
\end{proposition}

\begin{proof}
Assume there is no vertex of degree at least 4. Suppose we have a triangle $\triangle ABC$ enclosed by borderlines $AB$, $BC$, and $AC$. All vertices in the coloring must connected to at least three borderlines. This is because otherwise at each vertex there would be an angle greater than $\pi$, which is not possible in a triangle coloring, or the vertex would have degree at least 4, in which case we are done. We call $P$, $Q$, $R$ the vertices that are connected by borderlines to vertices $A$, $B$, $C$ respectively. Borderline $PA$ must lie on or between the lines that extend borderlines $AC$ and $AB$. This is because otherwise $PA$ would form an angle greater than $\pi$ with either $AC$ or $AB$, which is not possible in a triangle coloring. Likewise, we know that borderline $QB$ lies on or between lines extending borderlines $AB$ and $BC$ and that borderline $RC$ lies on or between lines extending borderlines $AC$ and $BC$.

We know that $PA$, $QB$, and $RC$ are each collinear with a borderline enclosing triangle $\triangle ABC$ because otherwise we can apply Lemma \ref{lem:obtuselem} to one borderline of $PA$, $QB$, and $RC$, a borderline enclosing $\triangle ABC$, and a different borderline of $PA$, $QB$, and $RC$, which forces a vertex of degree at least 4. If $PA$ and $QB$ are both collinear to $AB$, then we can apply Lemma \ref{lem:obtuselem} to borderlines $PA$, $AC$, and $RC$ or borderlines $QB$, $BC$, and $RC$, which forces a vertex of degree at least 4 on either borderline $AC$ or $BC$. The same reasoning can be applied to any pairs of $PA$, $QB$, and $RC$, so we know that no two of these borderlines are collinear with the same borderline enclosing $\triangle ABC$. This implies that borderlines $PA$, $QB$, $RC$ must each be collinear with a different borderline enclosing triangle $\triangle ABC$. Without loss of generality, suppose $PA$ is collinear with $AC$, $QB$ is collinear with $AB$, and $RC$ is collinear with $BC$ as shown in Figure \ref{fig:proptri}.

\begin{figure}[ht]
        \centering
\begin{tikzpicture}[line cap=round,line join=round,>=triangle 45,x=1.0cm,y=1.0cm, scale=1.1]
\draw [line width=1.pt] (2.,1.)-- (-2.,0.);
\draw [line width=1.pt] (-2.,0.)-- (0.,-2.);
\draw [line width=1.pt] (2.,1.)-- (0.,-2.);
\draw [line width=1.pt] (-2.753085958930627,0.7530859589306269)-- (2.9835294117647058,1.2458823529411764);
\draw [line width=1.pt] (0.,-2.) -- (-0.6923076923076921,-3.0384615384615383);
\draw [line width=1.pt] (-2.753085958930627,0.7530859589306269) -- (-3.93,1.93);
\draw [line width=1.pt] (2.9835294117647058,1.2458823529411764) -- (4.8,1.7);
\draw [line width=1.pt] (-2.753085958930627,0.7530859589306269)-- (-2.,0.);
\draw [line width=1.pt] (2.,1.)-- (2.9835294117647058,1.2458823529411764);
\begin{scriptsize}
\draw [fill=black] (2.,1.) circle (1pt);
\draw[color=black] (2.1,0.77) node {$B$};
\draw [fill=black] (-2.,0.) circle (1pt);
\draw[color=black] (-2.02,0.33) node {$A$};
\draw [fill=black] (0.,-2.) circle (1pt);
\draw[color=black] (0.18,-2.17) node {$C$};
\draw [fill=black] (-2.753085958930627,0.7530859589306269) circle (1pt);
\draw[color=black] (-2.7,0.99) node {$P$};
\draw [fill=black] (2.9835294117647058,1.2458823529411764) circle (1pt);
\draw[color=black] (2.9,1.49) node {$Q$};
\draw [fill=black] (-3.3,1.3) circle (1pt);
\draw[color=black] (-3.16,1.57) node {$X$};
\draw [fill=black] (3.7788235294117647,1.4447058823529408) circle (1pt);
\draw[color=black] (3.72,1.67) node {$Y$};
\draw [fill=black] (-0.36615384615384616,-2.5492307692307694) circle (1pt);
\draw[color=black] (-0.60,-2.35) node {$R$};
\end{scriptsize}
\end{tikzpicture}
\caption{This is an example of $\triangle ABC$.}\label{fig:proptri}
\end{figure}
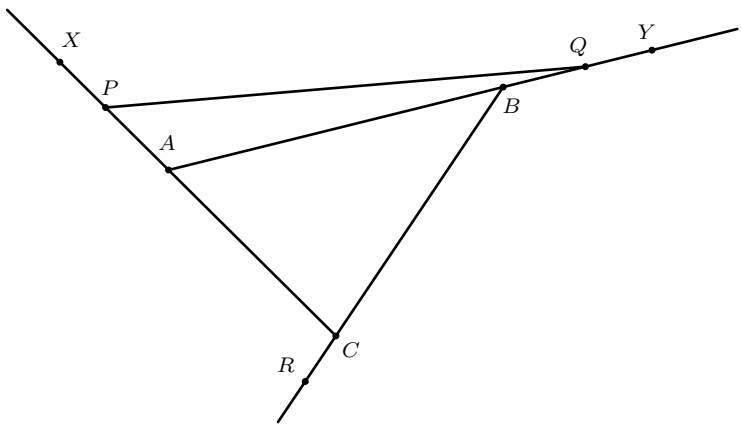

Vertex $P$ must be connected to two vertices besides $A$. Call them $X$ and $W$. Suppose neither borderline $XP$ nor $WP$ are collinear to $AC$. We can then apply Lemma \ref{lem:obtuselem} to $XP$, $PC$, and $RC$ or $XP$, $PA$, and $AB$ which forces a vertex of degree at least 4. The only exceptions to this are when $PX$ is collinear with $AC$, in which case we are done, or ray $PX$ intersects either ray $CR$ or ray $AB$. Thus, we consider the latter case. If both rays $PX$ and $PW$ intersect ray $CR$ or ray $AB$, then there would be an angle greater than $\pi$, which is not possible in a triangle coloring. Therefore, without loss of generality, we suppose ray $PX$ intersects ray $CR$ and $PW$ intersects ray $AB$. However, if this is the case, then $\angle XPW$ is greater than $\pi$, which is not possible in a triangle coloring. Hence, $P$ must be connected to a borderline collinear with $AC$. Hence, we suppose that $XP$ is collinear with $AC$. By a similar argument, we suppose a vertex $Y$ connected to $Q$ forms a borderline $YQ$ collinear with $AB$.

We know that $\angle PAB$ must lie within a triangle, which has vertices on the rays $CA$ and $AB$. Hence, the vertices of this triangle must form borderlines with $A$ that are each respectively collinear with $AC$ and $AB$. There must be borderlines connected to the vertices of this triangle collinear to $AC$ and $AB$ for the same reasoning applied to $XP$ and $YQ$. Thus, without loss of generality, suppose this triangle is $\triangle PAQ$. Applying Lemma \ref{lem:obtuselem} on borderlines $XP$, $PQ$, and $YQ$ forces a vertex of degree at least 4 on borderline $PQ$. Thus, any triangle coloring of the plane contains at least one vertex of degree at least 4. \end{proof}

Hence, by Theorem \ref{thm:genthm}, at least $7$ colors are needed to color the plane for any triangle coloring, which proves Theorem \ref{thm:trithm}.

\section{Discussion and open problems}

Since Theorem \ref{thm:genthm} requires the polygon coloring to have a vertex with degree at least 4, the only possible colorings that are left to be examined are those in which every vertex has degree at most 3. This type of coloring can not be analyzed using the same method. For instance, let $O$ be a vertex of degree 3 and $C_O$ be the unit circle centered at $O$. By Lemma \ref{lem:arclem} there is at least one crossing on $C_O$. Let $P$ be the crossing and $Q$ be a point one unit away on $C_O$. In the proof of Theorem  \ref{thm:genthm},  we can assume $Q$ is a crossing because otherwise seven colors are needed immediately. However, in a polygon coloring in which every vertex has degree at most 3, $Q$ does not have to be a crossing. This is because if $Q$ lies in a region, then it does not immediately imply that seven colors needed. Therefore, each crossing on $C_O$ can be a unit apart from pseudo-crossings and points within regions.

This leads us to ask for the fewest number of colors needed for polygon colorings in which every vertex has degree at most 3 (see Figure \ref{fig:tilings} for examples). If it is shown that at least seven colors are needed for this case, then the chromatic number of polygon colorings is 7 as a consequence of Theorem \ref{thm:genthm} and that would improve on Coulson's bound.

\begin{question}
What is the minimum number of colors needed for polygon colorings in which every vertex has degree at most 3?
\end{question}

Though Theorem \ref{thm:trithm} shows that triangle colorings require at least seven colors, this does not imply that the chromatic number of these colorings is $7$. The best coloring that we have found for triangle colorings (see Figure \ref{fig:besttri}) sets the upper bound for the chromatic number of triangle colorings to $8$. 

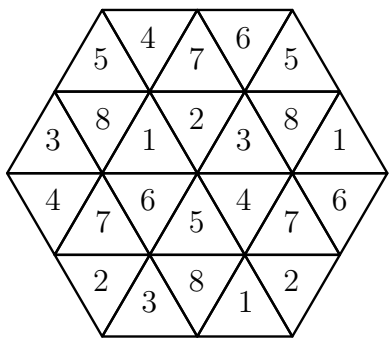
\begin{figure}[ht]
        \centering
                \begin{tikzpicture}[line cap=round,line join=round,>=triangle 45,x=1.0cm,y=1.0cm, scale=0.5]
\draw[line width=.9pt] (-1.25,2.165063509461097) -- (1.25,2.165063509461097) -- (0,4.330127018922194) -- cycle;
\draw[line width=.9pt] (1.25,2.165063509461097) -- (2.5,0) -- (3.75,2.1650635094610955) -- cycle;
\draw[line width=.9pt] (2.5,0) -- (1.25,-2.165063509461097) -- (3.75,-2.165063509461098) -- cycle;
\draw[line width=.9pt] (1.25,-2.165063509461097) -- (-1.25,-2.165063509461097) -- (0,-4.330127018922194) -- cycle;
\draw[line width=.9pt] (-1.25,-2.165063509461097) -- (-2.5,0) -- (-3.75,-2.1650635094610955) -- cycle;
\draw[line width=.9pt] (-2.5,0) -- (-1.25,2.165063509461097) -- (-3.75,2.165063509461098) -- cycle;
\draw[line width=.9pt] (-3.75,2.165063509461098) -- (-1.25,2.165063509461097) -- (-2.5,4.3301270189221945) -- cycle;
\draw[line width=.9pt] (1.25,2.165063509461097) -- (3.75,2.1650635094610955) -- (2.5,4.330127018922194) -- cycle;
\draw[line width=.9pt] (3.75,-2.165063509461098) -- (1.25,-2.165063509461097) -- (2.5,-4.3301270189221945) -- cycle;
\draw[line width=.9pt] (-1.25,-2.165063509461097) -- (-3.75,-2.1650635094610955) -- (-2.5,-4.330127018922194) -- cycle;
\draw[line width=.9pt] (0,4.330127018922194) -- (-2.5,4.3301270189221945) -- (-1.25,2.165063509461099) -- cycle;
\draw[line width=.9pt] (2.5,4.330127018922194) -- (0,4.330127018922194) -- (1.25,2.1650635094610955) -- cycle;
\draw[line width=.9pt] (0,-4.330127018922194) -- (2.5,-4.3301270189221945) -- (1.25,-2.165063509461099) -- cycle;
\draw[line width=.9pt] (-2.5,-4.330127018922194) -- (0,-4.330127018922194) -- (-1.25,-2.1650635094610955) -- cycle;
\draw[line width=.9pt] (2.5,0) -- (3.75,-2.165063509461098) -- (5,0) -- cycle;
\draw[line width=.9pt] (2.5,0) -- (5,0) -- (3.75,2.1650635094610973) -- cycle;
\draw[line width=.9pt] (-2.5,0) -- (-3.75,2.165063509461098) -- (-5,0) -- cycle;
\draw[line width=.9pt] (-2.5,0) -- (-5,0) -- (-3.75,-2.1650635094610973) -- cycle;
\draw[line width=.9pt] (-1.25,2.165063509461097) -- (0,0) -- (1.25,2.165063509461097) -- cycle;
\draw[line width=.9pt] (-2.5,0) -- (0,0) -- (-1.25,2.165063509461097) -- cycle;
\draw[line width=.9pt] (-2.5,0) -- (-1.25,-2.165063509461097) -- (0,0) -- cycle;
\draw[line width=.9pt] (0,0) -- (2.5,0) -- (1.25,2.165063509461097) -- cycle;
\draw[line width=.9pt] (0,0) -- (1.25,-2.165063509461097) -- (2.5,0) -- cycle;
\draw[line width=.9pt] (-1.25,-2.165063509461097) -- (1.25,-2.165063509461097) -- (0,0) -- cycle;

\draw (-1.7575, 1.4500635095) node[anchor=north west] {1};
\draw (-.515, 1.9950635095) node[anchor=north west] {2};
\draw (0.7275, 1.4500635095) node[anchor=north west] {3};

\draw (-1.7875, -.1500635095) node[anchor=north west] {6};
\draw (-.515, -.7200635095) node[anchor=north west] {5};
\draw (0.7275, -.1500635095) node[anchor=north west] {4};

\draw (1.9805, -.7200635095) node[anchor=north west] {7};
\draw (1.9805, 1.9950635095) node[anchor=north west] {8};

\draw (-2.97525, -.7200635095) node[anchor=north west] {7};
\draw (-2.97525, 1.9950635095) node[anchor=north west] {8};

\draw (-3.025, 3.5950635095) node[anchor=north west] {5};
\draw (-.515, 3.5950635095) node[anchor=north west] {7};
\draw (1.9805, 3.5950635095) node[anchor=north west] {5};

\draw (-3.025, -2.2950635095) node[anchor=north west] {2};
\draw (-.515, -2.2950635095) node[anchor=north west] {8};
\draw (1.9805, -2.2950635095) node[anchor=north west] {2};

\draw (-1.7575, -2.850635095) node[anchor=north west] {3};
\draw (0.7775, -2.850635095) node[anchor=north west] {1};

\draw (-1.7875, 4.150635095) node[anchor=north west] {4};
\draw (0.71475, 4.150635095) node[anchor=north west] {6};

\draw (-4.28775, 1.4500635095) node[anchor=north west] {3};
\draw (-4.28775, -.1500635095) node[anchor=north west] {4};

\draw (3.252775, 1.4500635095) node[anchor=north west] {1};
\draw (3.252775, -.1500635095) node[anchor=north west] {6};

\end{tikzpicture}
\caption{Each triangle has unit side length and every vertex and border point is given the color of the region directly above it.}\label{fig:besttri}
\end{figure}

Since this coloring uses more than seven colors, naturally one asks what the chromatic number is for triangle colorings. We conjecture that the chromatic number for such colorings is $8$.

\begin{question}
What is the chromatic number of triangle colorings of the plane?
\end{question}

Finally, we note that it may be possible to extend our approach to polygon-like colorings that use differentiable curves instead of line segments. In such a coloring, we can approximate the curved borders connected at vertices by the tangent lines of those curves, which allows us to use most of the arguments for colorings with straight line borders. It would thus follow that a coloring with regions that are curved triangles also requires seven colors. However, one complication that would have to be dealt with is that some borders could partially coincide with circle arcs.

\section*{Acknowledgements}
I am incredibly grateful for the guidance Frank de Zeeuw has given me throughout the process of developing this paper. Also, I would like to thank Adam Sheffer and Dan Stefanica for giving me the opportunity to work on this project and an anonymous reviewer for pointing out an error in an earlier version of Section 5. All the diagrams in this paper have been made with GeoGebra.


\begin{thebibliography}{99}

\bibitem{Cou}
D. Coulson,
\emph{On the chromatic number of plane tilings},
Journal of the Australian Mathematical Society {\bf 77}, 191--196, 2004.

\bibitem{Gre}
A. de Grey,
\emph{The chromatic number of the plane is at least 5},
Geombinatorics {\bf 28}, 18--31, 2018.

\bibitem{Exo}
G. Exoo and D. Ismailescu,
\emph{The Chromatic Number of the Plane is At Least 5: A New Proof},
Discrete \& Computational Geometry {\bf 64}, 216--226, 2020.

\bibitem{Fal}
K.J. Falconer, 
\emph{The realization of distances in measurable subsets covering $\R^n$},
Journal of Combinatorial Theory, Series A {\bf 31}, 184--189, 1981.

\bibitem{Mou}
J. Moustakis,
\emph{On the chromatic number of the plane and square colorings},
Master thesis at EPFL, Lausanne, Switzerland, 2016.

\bibitem{Soi}
A. Soifer,
\emph{The mathematical coloring book},
Springer, New York, 2009.

\bibitem{Tho}
C. Thomassen,
\emph{On the Nelson Unit Distance Coloring Problem},
The American Mathematical Monthly {\bf 106}, 850--853, 1999.

\bibitem{Tow}
S.P. Townsend,
\emph{Every 5-coloured map in the plane contains a monochrome unit},
Journal of Combinatorial Theory, Series A {\bf 30}, 114--115, 1981.

\bibitem{Woo}
D.R. Woodall,
\emph{Distances realized by sets covering the plane},
Journal of Combinatorial Theory, Series A {\bf 14}, 187--200, 1973.

\end{thebibliography}
\end{document}